\documentclass[10pt, leqno]{amsart}
\usepackage{amssymb}
\usepackage{amsmath}
\usepackage[cmtip,all]{xy}
\usepackage{stmaryrd}

\theoremstyle{plain}
\newtheorem{Lem}{Lemma}[section]

\newtheorem{Cor}[Lem]{Corollary}
\newtheorem{Thm}[Lem]{Theorem}

{\theoremstyle{definition} 

\newtheorem{Rk}[Lem]{Remark}
\newtheorem{Def}[Lem]{Definition}} 
 
\newcommand{\zig}{\addtocounter{Lem}{1}\tag{\theLem}} 
 
\DeclareMathOperator*{\holim}{holim}

\DeclareMathOperator*{\colim}{colim}

\def\:{\colon}

\DeclareMathAlphabet{\mathpzc}{OT1}{pzc}{m}{it}

\begin{document}

\title{Obtaining intermediate rings 
of a local profinite Galois extension without localization}
\author{Daniel G. Davis$\sp 1$}
\begin{abstract}
Let $E_n$ be the Lubin-Tate spectrum and let $G_n$ be the $n$th extended 
Morava stabilizer group. Then there is a 
discrete $G_n$-spectrum $F_n$, 
with $L_{K(n)}(F_n) \simeq E_n$, that has the property that $(F_n)^{hU} 
\simeq E_n^{hU}$, for every open subgroup $U$ of $G_n$. In particular, 
$(F_n)^{hG_n} \simeq L_{K(n)}(S^0).$ More generally, for any closed 
subgroup $H$ of $G_n$, there is a discrete $H$-spectrum $Z_{n, H}$, such 
that $(Z_{n, H})^{hH} \simeq E_n^{hH}.$ These conclusions are obtained from  
results about consistent $k$-local profinite $G$-Galois extensions $E$ of finite 
vcd, where 
$L_k(-)$ is $L_M(L_T(-))$, with $M$ a finite spectrum and $T$ smashing. 
For example, we show that $L_k(E^{hH}) \simeq E^{hH}$, for every open 
subgroup $H$ of $G$. 
\end{abstract}

\footnotetext[1]{The author was supported by a grant 
from the Louisiana Board of Regents Support Fund.}
 
\maketitle

\section{Introduction}
\par
Let $n \geq 1$ and let $p$ be a prime. Let 
$K(n)$ be the $n$th Morava $K$-theory spectrum and let 
$G_n = S_n \rtimes \mathrm{Gal}(\mathbb{F}_{p^n}/\mathbb{F}_p)$ 
be the nth extended Morava stabilizer group. Also, let 
$E_n$ be the $n$th Lubin-Tate spectrum, with
$\pi_\ast(E_n)=W(\mathbb{F}_{p^n})\llbracket u_1, ..., 
u_{n-1}\rrbracket[u^{\pm 1}],$
where $W(\mathbb{F}_{p^n})$ is the ring of Witt vectors 
with coefficients in 
the field $\mathbb{F}_{p^n}$, the degree of $u$ is $-2$, 
and the complete power series ring 
is in degree zero. 
\par
By \cite{cts}, the profinite group $G_n$ acts continuously on $E_n$, 
so that for every closed subgroup $H$ of $G_n$, 
$E_n$ is a continuous $H$-spectrum, and the homotopy fixed point 
spectrum $E_n^{hH}$ can be formed. Also, by \cite{thesis} (see 
\cite[Theorem 8.2.1]{joint} for a more efficient proof), there is an equivalence 
$E_n^{hH} \simeq E_n^{dhH}$, where $E_n^{dhH}$ is the 
$K(n)$-local commutative $S^0$-algebra of \cite{DH} and, for open normal 
subgroups $H$, a key ingredient in 
building $E_n$ as a continuous $G_n$-spectrum.
\par
In more detail, as in 
\cite[\S 5.2]{joint}, let $\mathpzc{Alg}$ be the model category 
of discrete commutative $G_n$-$L_{K(n)}(S^0)$-algebras: 
objects of $\mathpzc{Alg}$ are 
discrete $G_n$-spectra that are also commutative $L_{K(n)}(S^0)$-algebras, 
and morphisms are $G_n$-equivariant maps of commutative 
$L_{K(n)}(S^0)$-algebras. Let $(-)_\mathpzc{F} \: \mathpzc{Alg} \rightarrow 
\mathpzc{Alg}$ be a fibrant replacement functor, 
and let $U <_o G_n$ denote an 
open subgroup of $G_n$. 
Then 
set \[F_n = \colim_{N \vartriangleleft_o G_n} (E_n^{dhN})_\mathpzc{F}.\] 
Since each $E_n^{dhN}$ is a $G_n/N$-spectrum that is $K(n)$-local, 
$F_n$ is 
a discrete $G_n$-spectrum that is $E(n)$-local and, 
by \cite[Lemma 6.7]{cts}, {\em not} $K(n)$-local. 
Here, $E(n)$ is 
the Johnson-Wilson spectrum, with 
$\pi_\ast(E(n)) = \mathbb{Z}_{(p)}[v_1, ..., v_{n-1}][v_n^{\pm1}]$.
\par
Given a tower $M_0 \leftarrow M_1 \leftarrow \cdots 
\leftarrow M_i \leftarrow \cdots$ 
of generalized Moore spectra, such that for any $E(n)$-local spectrum 
$Z$, $L_{K(n)}(Z) \simeq \holim_i (Z \wedge M_i)$ (as in 
\cite[\S 2]{HoveyCech}), 
\[E_n \simeq L_{K(n)}(F_n) \simeq \holim_i (F_n \wedge M_i),\] where 
each $F_n \wedge M_i$ is a discrete $G_n$-spectrum. Hence, as in 
\cite{cts}, for any closed subgroup $H$ of $G_n$,
\[E_n^{hH} = (\holim_i (F_n \wedge M_i))^{hH} 
= \holim_i (F_n \wedge M_i)^{hH} \simeq L_{K(n)}((F_n)^{hH}),\]
where the last step uses that each $M_i$ is a finite spectrum.
In particular, when $H=G_n$, we have 
\[L_{K(n)}(S^0) \simeq E_n^{hG_n} \simeq L_{K(n)}((F_n)^{hG_n}),\] 
where the first equivalence is given by \cite[Theorem 1, (iii)]{DH} and 
\cite[Corollary 8.1.3]{joint}. When $p=2$, $G_2$ has a finite subgroup 
$\mathcal{G}_{48}$ of order 48, and the Hopkins-Miller spectrum $EO_2$ is 
given by
\[EO_2 = E_2^{h\mathcal{G}_{48}} \simeq L_{K(2)}((F_2)^{h\mathcal{G}_{48}})\] 
(for more on $EO_2$, see \cite[Theorem 5.1]{eo2homotopy}). Also, 
for each $n$ and $p$, there is a finite subgroup $\mathbb{F}_{p^n}^\times$ of 
$S_n$ such that, by setting $K_n = \mathbb{F}_{p^n}^\times \rtimes 
\mathrm{Gal}(\mathbb{F}_{p^n}/\mathbb{F}_p),$
\[\smash{\widehat{E(n)}} \mathrel{\mathop:}= 
L_{K(n)}(E(n)) \simeq E_n^{h{K_n}} \simeq 
L_{K(n)}((F_n)^{hK_n})\] 
(see the proof of \cite[Proposition 5.4.9, \negthinspace (a)]{rognes}). 
\par
The equivalence $E_n^{hH} \simeq L_{K(n)}((F_n)^{hH})$ and the 
above three examples have the following pattern in common:  the $K(n)$-local 
spectrum of interest is obtained by first 
taking the homotopy fixed points of the non-$K(n)$-local 
discrete $G_n$-spectrum 
$F_n$, and then $K(n)$-localizing. 
As seen above in the case of $E_n^{hH}$, this can be equivalently 
expressed by saying that the $K(n)$-local spectrum is obtained by taking the 
homotopy fixed points of the homotopy limit of the tower $\{F_n \wedge M_i\}_i$ 
of discrete $G_n$-spectra. The strength of this pattern is such 
that it might seem 
that the use of a 
tower of discrete $G_n$-spectra, instead of just a single 
discrete $G_n$-spectrum, is a necessary component of it. 
Similarly, one might conclude from the above discussion that the second step of 
$K(n)$-localizing the homotopy fixed points is an integral part of the pattern. 
However, in this paper, we show that such conclusions are not correct. 
\par
For any open subgroup $U$ of $G_n$, it turns out that 
\[E_n^{hU} \simeq (F_n)^{hU}.\] Thus, we have
\[(F_n)^{hG_n} \simeq L_{K(n)}(S^0),\] which is an equivalence 
that was first obtained by Mark Behrens. 
Hence, for any finite 
spectrum $X$, since $L_{K(n)}(X) \simeq L_{K(n)}(S^0) \wedge X$, 
\[L_{K(n)}(X) \simeq (F_n \wedge X)^{hG_n},\]  
where $F_n \wedge X$ is a discrete $G_n$-spectrum, with $G_n$ acting 
trivially on $X$. Therefore, the 
aforementioned tower/``$K(n)$-localizing at the end" is not necessary, and 
a single discrete $G_n$-spectrum ($F_n$ or $F_n \wedge X$) 
suffices to yield 
the $K(n)$-local spectra $E_n^{hU}$, $L_{K(n)}(S^0)$, and $L_{K(n)}(X),$ for 
finite $X$.
\par
Also, for any closed 
subgroup $H$ of $G_n$, we show that there is a discrete $H$-spectrum $Z_{n,H}$ 
such that 
\[E_n^{hH} \simeq (Z_{n,H})^{hH},\] so that, for example, 
\[EO_2 \simeq (Z_{2, \mathcal{G}_{48}})^{h\mathcal{G}_{48}} \ \ \mathrm{and} \ \ 
\widehat{E(n)} \simeq (Z_{n, K_n})^{hK_n}.\]
Thus, once again, to obtain the 
$K(n)$-local spectrum of interest, it suffices to just take the homotopy 
fixed points of a single 
discrete $H$-spectrum.     
\par
We obtain the above results by considering the more general context of 
$k$-local profinite Galois extensions. Here, 
following \cite[Assumption 1.0.3]{joint}, $k$ denotes a 
spectrum whose localization functor $L_k(-)$ is a 
composite $L_M(L_T(-))$ of localizations, 
where $M$ is a finite spectrum and 
$T$ is smashing. (Also, as explained in 
Section 2, we make the technical assumption that $k$ is an $S$-cofibrant 
symmetric spectrum.) Examples of such 
spectra $k$ include $S^0$, $H\mathbb{F}_p$, $E(n)$, and $K(n)$ 
(see \cite{Bousfieldlocal, HoveyCech}). 
\par
Now let $A$ be a $k$-local cofibrant commutative symmetric ring spectrum, 
$E$ a commutative $A$-algebra, and $G$ a profinite group.
Then $E$ is a \emph{$k$-local profinite $G$-Galois extension} of 
$A$ if
\begin{enumerate}
\item[(a)] there is a directed system of 
finite $k$-local $G/U_\alpha$-Galois extensions
$E_\alpha$ of $A$, where $\{U_\alpha\}_\alpha$ is a cofinal 
system of open normal subgroups of $G$ (see \cite[Section 4.1]{rognes} for 
the definition of a finite $k$-local Galois extension);
\item[(b)] 
the commutative $A$-algebra $E$ is given by 
\[E = \colim_\alpha (E_\alpha)_{fGA},\] where 
\[(-)_{fGA} \: \mathrm{Alg}_{A,G} \rightarrow \mathrm{Alg}_{A,G}\] is 
a fibrant replacement functor for the model category $\mathrm{Alg}_{A,G}$ of 
discrete commutative $G$-$A$-algebras (see \cite[Section 5.2]{joint});
\item[(c)]
each of the maps $E_\alpha \rightarrow E_\beta$
is $G$-equivariant and is a cofibration of underlying 
commutative $A$-algebras; and
\item[(d)] 
for each $\alpha \le \beta$, 
the natural map $E_\alpha \rightarrow (E_\beta)^{h(U_\alpha/U_\beta)}$ 
is a weak equivalence.
\end{enumerate} Also, the Galois extension $E$ has {\em finite vcd} 
(finite virtual cohomological dimension) if $G$ has finite vcd (that is, 
there exists an open 
subgroup $U$ of $G$ and a positive integer $m$ such that the 
continuous cohomology $H^s_c(U; M) = 0,$ 
for all $s > m$, whenever $M$ is a discrete $U$-module).
\par
Given a $k$-local profinite $G$-Galois extension $E$ of $A$, the $k$-local 
Amitsur derived completion $A^{\wedge}_{k,E}$ is the homotopy limit of 
the cosimplicial spectrum
\[\xymatrix@C.25in{L_k(E) \ar@<.5ex>[r] \ar@<-.5ex>[r] & 
L_k(E \wedge_A E) \ar@<.7ex>[r] \ar[r] \ar@<-.7ex>[r] &
L_k(E \wedge_A E \wedge_A E) \ar@<.8ex>[r] \ar@<.25ex>[r] \ar@<-.25ex>[r] 
\ar@<-.8ex>[r] & \cdots}\] that is built from the unit map $A \rightarrow E$ and the 
multiplication $E \wedge_A E \rightarrow E$ (see \cite[Definition 8.2.1]{rognes} 
and \cite{derivedcompletion}). Then, 
following \cite[Definition 1.0.4, \negthinspace (1)]{joint}, the extension $E$ is 
{\em consistent} if the coaugmentation $A \rightarrow A^{\wedge}_{k,E}$ is a 
weak equivalence.
\par
Suppose that 
$E$ is a consistent $k$-local profinite $G$-Galois extension of $A$ of finite vcd. 
Then in Theorem \ref{galois}, for each open subgroup $U$ of $G$, we show that 
\[L_k(E^{hU}) \simeq E^{hU},\] and, by Corollary \ref{A}, 
\[A \simeq E^{hG}.\] More generally, 
if $H$ is any closed subgroup of $G$, then, by (\ref{coarse}), 
there is a discrete $H$-spectrum $Z_{E,H}$ such that 
\[L_k(E^{hH}) \simeq (Z_{E,H})^{hH}.\] If the Galois 
extension $E$ is profaithful (in the sense of \cite{joint}; see Definition 
\ref{profaithful}), then the dependence of $Z_{E,H}$ on $H$ can be lessened 
somewhat (see Theorem \ref{lasttheorem}).
\par
Since $F_n$ 
is a consistent profaithful $K(n)$-local profinite $G_n$-Galois 
extension of $L_{K(n)}(S^0)$ of finite vcd (by 
\cite[Proposition 8.1.2]{joint} and 
\cite[Theorem 5.4.4, 
Proposition 5.4.9,~ \negthinspace(b)]{rognes}), the 
results stated earlier for $E_n$ follow immediately 
from the results described in the previous paragraph. 
\par
We briefly summarize the remaining portions of this paper. Let 
$G$ be a profinite group and let $H$ be a closed subgroup of $G$. 
In Section \ref{sectiontwo}, we recall various definitions and results about 
(local) homotopy fixed points that are used later. In Section \ref{useful}, we 
study several discrete $H$-spectra that are canonically associated to a 
discrete $G$-spectrum $X$ and are useful for the results 
of Section \ref{sectionfour}, where we show that, when $X$ is 
$T$-local and $G$ has finite vcd, $L_k(X^{hH})$ can be obtained by 
just taking the homotopy fixed points of a single discrete $H$-spectrum.
\par
We conclude the Introduction with some comments about the terminology that we use. 
All spectra are symmetric spectra of simplicial sets 
and we use $\Sigma\mathrm{Sp}$ 
to denote the model 
category of spectra (more precisely, the stable model category of symmetric 
spectra, as defined in \cite[Section~3.4]{HSS}). 
\par
Given a profinite group 
$G$, a discrete $G$-spectrum (following \cite[Section~2.3]{joint}; 
see also \cite[Section~3]{cts}) is a spectrum such that, for each $j \geq 0$, the 
simplicial set $X_j$ is a pointed 
simplicial discrete $(G \times \Sigma_j)$-set, where $\Sigma_j$ denotes the 
$j$th symmetric group, together with compatible
$(G \times \Sigma_{l} \times \Sigma_j)$-equivariant maps
\[ \sigma^l \: S^l \wedge X_j \rightarrow X_{j+l}. \]
Here, $S^l = (S^1)^{\wedge l}$ 
has the trivial $G$-action and the factors in this smash product are 
permuted by the action of $\Sigma_l$.  
\par
We use 
$\Sigma\mathrm{Sp}_G$ 
to denote the category of discrete $G$-spectra, in which a 
morphism $f \: X \rightarrow Y$ of
discrete $G$-spectra is a collection of 
$(G \times \Sigma_j)$-equivariant maps $f_j \: X_j \rightarrow Y_j$ 
of pointed simplicial sets 
which are compatible with the structure maps $\sigma^1$. 
\par
All filtered colimits in this paper are taken in the category $\Sigma\mathrm{Sp}$. 
This statement holds even for part (b) in the definition of $k$-local profinite 
$G$-Galois extension that we gave earlier, since 
filtered colimits in the category 
of commutative $A$-algebras are formed in 
$\Sigma\mathrm{Sp}$ (a reference for this fact is 
\cite[Lemma 5.3.4]{joint}).
\vspace{.1in}
\par
\noindent
\textbf{Acknowledgements.} I thank Mark Behrens for helpful discussions 
about $k$-local homotopy fixed points and $F_n$, and for 
realizing that 
$(F_n)^{hG_n} \simeq L_{K(n)}(S^0)$. Also, I thank Ethan Devinatz, for 
helpful comments about $E_n^{hN}$, 
and the referee, for encouraging 
remarks.
\section{Preliminaries on (local) homotopy fixed points}\label{sectiontwo}
\par
In this section, we recall various definitions and results about 
(local) homotopy fixed points that will be useful in the rest of the paper.  
\par
Let $G$ be a profinite group. By \cite[Theorem 2.3.2]{joint}, there is a model 
category structure on $\Sigma\mathrm{Sp}_G$, where a morphism $f$ is a 
weak equivalence (cofibration) if and only if $f$ is a weak equivalence 
(cofibration) in $\Sigma\mathrm{Sp}$.  Let \[(-)_{fG} \: \Sigma\mathrm{Sp}_G 
\rightarrow \Sigma\mathrm{Sp}_G\] be a fibrant replacement functor, so 
that if $X$ is a discrete $G$-spectrum, there is a natural map 
$X \rightarrow X_{fG}$ that is a trivial cofibration, with $X_{fG}$ fibrant, in 
$\Sigma\mathrm{Sp}_G$. Since the fixed points functor 
\[(-)^G \: \Sigma\mathrm{Sp}_G \rightarrow \Sigma\mathrm{Sp}\] is a 
right Quillen functor (by \cite[Lemma 3.1.1]{joint}), the homotopy fixed 
points $(-)^{hG}$ are defined to be the right derived functor of $(-)^G$, so 
that, given $X \in \Sigma\mathrm{Sp}_G$, 
\[X^{hG} = (X_{fG})^G\] (see \cite[Section 3.1]{joint}).
\par
Now let $F$ be any spectrum. Recall from \cite[Section 5.3]{HSS} that, 
by using cofibrant replacement in the $S$ model structure on the category 
of symmetric spectra, there is a 
weak equivalence $F_c \rightarrow F$ in $\Sigma\mathrm{Sp}$ (with the 
usual model structure, as described near the end of the Introduction) such that 
the spectrum $F_c$ is $S$-cofibrant. For this reason, we always assume that 
the original spectrum $F$ itself is $S$-cofibrant. This assumption is used, 
for example, in the proof of Lemma \ref{fibrant} and the discussion after this 
proof. 
\par 
The category of spectra, when equipped with the $F$-local model category 
structure, gives the model category 
$(\Sigma\mathrm{Sp})_F$ (considered, for example, in \cite[Section 6.1]{joint}). 
In $(\Sigma\mathrm{Sp})_F$, a morphism $f$ of spectra 
is a weak equivalence 
(cofibration) if and only if $f$ is an $F$-local equivalence (cofibration) of 
spectra. Here, $f$ is an $F$-local equivalence exactly when $F \wedge f$ is a 
weak equivalence in $\Sigma\mathrm{Sp}$. 
We define the Bousfield localization functor 
\[L_F(-) \: (\Sigma\mathrm{Sp})_F \rightarrow (\Sigma\mathrm{Sp})_F\] by taking 
it to be a 
fibrant replacement functor for $(\Sigma\mathrm{Sp})_F$, so that, given a spectrum 
$X$, the natural map $X \rightarrow L_F(X)$ is a trivial cofibration, with $L_F(X)$ 
fibrant, in $(\Sigma\mathrm{Sp})_F$. 
\par
In \cite[Section 6.1]{joint}, it was shown that the category $\Sigma\mathrm{Sp}_G$ 
can be equipped with an $F$-local model category structure, to give 
the model category 
$(\Sigma\mathrm{Sp}_G)_F$, in which a morphism $f$ of discrete 
$G$-spectra is a weak equivalence (cofibration) if and only 
if $f$ is a weak equivalence (cofibration) in $(\Sigma\mathrm{Sp})_F$. 
(The idea for this $F$-local model structure has an antecedent in 
\cite[Section 7]{hGal}, where such a model structure is put on the 
category of simplicial discrete $G$-sets.) 
Additionally, 
the fixed points functor \[(-)^G \: (\Sigma\mathrm{Sp}_G)_F 
\rightarrow (\Sigma\mathrm{Sp})_F\]
is a right Quillen functor (see \cite[Section 6.1]{joint} for more detail). 
This implies, for example, that 
if $X$ is a discrete $G$-spectrum that 
is fibrant in $(\Sigma\mathrm{Sp}_G)_F$, then 
$X^G$ is fibrant in $(\Sigma\mathrm{Sp})_F$, and hence, $X^G$ is an 
$F$-local spectrum. 
\par
Let \[(-)_{f_FG} \: (\Sigma\mathrm{Sp}_G)_F \rightarrow 
(\Sigma\mathrm{Sp}_G)_F\] be a fibrant replacement functor, so that, given 
a discrete $G$-spectrum $X$, 
there is a natural map $X \rightarrow X_{f_FG}$ that 
is a trivial cofibration, with $X_{f_FG}$ fibrant, 
in $(\Sigma\mathrm{Sp}_G)_F$. Then, as in 
\cite[Section 6.1]{joint}, the $F$-local homotopy fixed points 
$(-)^{h_FG}$ are defined to be the right derived functor of $(-)^G$ with 
respect to the $F$-local model structure, so that, given a discrete $G$-spectrum 
$X$, the $F$-local homotopy fixed 
point spectrum $X^{h_FG}$ is given by \[X^{h_FG} = (X_{f_FG})^G.\] Notice 
that $X^{h_FG}$ is an $F$-local spectrum.
\par
Though the 
following result was (quickly) noted in the proof of \cite[Proposition 6.1.7, (3)]{joint}, 
because of its usefulness, we present it as a lemma and give a proof. 
\begin{Lem}\label{fibrant}
Let $G$ be a profinite group and let $X$ be a discrete $G$-spectrum. 
If $X$ is fibrant in $(\Sigma\mathrm{Sp}_G)_F$, then $X$ is also fibrant 
in $\Sigma\mathrm{Sp}_G$.
\end{Lem}
\begin{proof} Let $f$ be a trivial cofibration 
in $\Sigma\mathrm{Sp}_G$, so that 
$f$ is both a weak equivalence and a cofibration of spectra. Since 
$F$ is $S$-cofibrant, \cite[Corollary 5.3.10]{HSS} implies that the 
functor $F \wedge (-)$ preserves weak equivalences in $\Sigma\mathrm{Sp}$, so that 
$f$ is an $F$-local equivalence. Thus, $f$ is a trivial cofibration 
in $(\Sigma\mathrm{Sp}_G)_F$. This conclusion implies that in 
$\Sigma\mathrm{Sp}_G$, the map 
$X \rightarrow \ast$ to the terminal object 
has the right lifting property with respect to all trivial 
cofibrations.
\end{proof}
\par
Let $X$ be a discrete $G$-spectrum. 
By Lemma \ref{fibrant}, 
$X_{f_FG}$ is fibrant in $\Sigma\mathrm{Sp}_G$, and hence, 
there exists a morphism \[f_X^G \: X_{fG} \rightarrow X_{f_FG}\] 
in $\Sigma\mathrm{Sp}_G.$ Notice that since $F$ is $S$-cofibrant, the map 
$X \rightarrow X_{fG}$ is an $F$-local equivalence, so that the map 
$f_X^G$ also is an $F$-local equivalence. Taking 
the $G$-fixed points of $f_X^G$ yields the map 
\[X^{hG} = (X_{fG})^G \rightarrow (X_{f_FG})^G = X^{h_FG}.\] Therefore, since 
$X^{h_FG}$ is fibrant 
in $(\Sigma\mathrm{Sp})_F$, the trivial cofibration 
$X^{hG} \rightarrow L_F(X^{hG})$ in $(\Sigma\mathrm{Sp})_F$ induces 
a map 
\begin{equation}\zig\label{map}
L_F(X^{hG}) \rightarrow X^{h_F G}.
\end{equation}  
\par 
When $F = k$ (where $k$ is defined as in the Introduction and is assumed to be 
$S$-cofibrant), the following result gives 
conditions under which the map in (\ref{map}) is a weak equivalence. 
\begin{Thm}[{\cite[Proposition 6.1.7, (3)]{joint}}]\label{key}
Let $G$ have finite vcd and suppose that 
$X$ is a discrete $G$-spectrum. If $X$ is $T$-local, then 
the map $L_k(X^{hG}) \rightarrow X^{h_k G}$ is a weak equivalence 
in $\Sigma\mathrm{Sp}$. 
\end{Thm}
\par
Given an open subgroup $U$ of $G$, let 
\[\mathrm{Res}^U_G \: \Sigma\mathrm{Sp}_G \rightarrow \Sigma\mathrm{Sp}_U, 
\ \ \ \mathrm{Res}^U_G(X) = X,\] be the functor 
that takes a discrete $G$-spectrum $X$ and 
regards it as a discrete $U$-spectrum. The proof of \cite[Proposition 3.3.1, (2)]{joint} 
shows that this functor preserves fibrant objects. Similarly, there is a 
functor \[(\mathrm{Res}^U_G)_F \: (\Sigma\mathrm{Sp}_G)_F 
\rightarrow (\Sigma\mathrm{Sp}_U)_F\] that regards a discrete $G$-spectrum 
as a discrete $U$-spectrum, and the proof of \cite[Proposition 6.1.7, (1)]{joint} 
shows that this functor preserves fibrant objects. We summarize these remarks 
in the following result. 
\begin{Lem}[{\cite{joint}}]\label{open}
Let $G$ be a profinite group, $U$ an open subgroup of $G$, and $X$ a discrete 
$G$-spectrum. If $X$ is fibrant in $\Sigma\mathrm{Sp}_G$, then it is fibrant 
in $\Sigma\mathrm{Sp}_U$. Similarly, if $X$ is fibrant in $(\Sigma\mathrm{Sp}_G)_F$, 
then it is fibrant in $(\Sigma\mathrm{Sp}_U)_F.$
\end{Lem}
\par
We conclude this section with several useful facts about spectra. 
Suppose that $\{g_\alpha \: Y_\alpha \rightarrow Z_\alpha\}_\alpha$ is a 
filtered system of weak equivalences of spectra, such that each 
$Y_\alpha$ and $Z_\alpha$ is fibrant. By \cite[Corollaries 3.4.13, 3.4.16]{HSS}, 
the spectra $Y_\alpha$ and $Z_\alpha$ 
consist of Kan complexes, and hence, 
the map $\pi_\ast(\colim_\alpha g_\alpha)$ is an 
isomorphism, so that $\colim_\alpha g_\alpha$ is a weak equivalence (by 
\cite[Theorem 3.1.11]{HSS}). 
\par
If $f$ is a trivial cofibration in $\Sigma\mathrm{Sp}$, 
then $f$ is also a trivial cofibration in $(\Sigma\mathrm{Sp})_F$. 
This fact implies that, 
given a spectrum $X$, the localization $L_F(X)$ is fibrant in $\Sigma\mathrm{Sp}$.
\section{Useful discrete $H$-spectra associated to a 
discrete $G$-spectrum}\label{useful}
\par
In this section, we restrict ourselves to $F = k$, since this is all that is needed 
in later sections. As before, $G$ is a profinite group, and we let $X$ be a 
discrete $G$-spectrum. Given a closed subgroup 
$H$ of $G$, we define and study 
two discrete $H$-spectra that are useful for our later results. 
\par
Let $N$ and $N'$ be proper open 
normal subgroups of $G$, with $N$ a subgroup of $N'$. Since the map 
$X \rightarrow X_{fN'}$ is a trivial cofibration in 
$\Sigma\mathrm{Sp}_N$, the right lifting property in $\Sigma\mathrm{Sp}_N$ 
of the map 
$X_{fN} \rightarrow \ast$  yields 
an $N$-equivariant map 
\[\lambda_{N,N'} \: X_{fN'} \rightarrow X_{fN}.\] The map $\lambda_{N,N'}$ 
induces the 
map \[\lambda^N_{N'} \: X^{hN'} \rightarrow X^{hN},\] 
which is defined to be the composition 
\[X^{hN'} = (X_{fN'})^{N'} \hookrightarrow (X_{fN'})^N \rightarrow (X_{fN})^N = X^{hN}.\] 
\par
Notice that 
the spectra $X^N$ and $(X_{fG})^N$ have natural $G/N$-actions, and hence, 
natural $G$-actions, through the projection $G \rightarrow G/N$. 
However, in general, 
the spectrum $X^{hN} = (X_{fN})^N$ 
is not known to have a natural $G/N$-action, 
apart from the trivial one, since $X_{fN}$ is not known, in general, 
to have a $G$-action. This fact implies that, unlike in the case of the 
diagram $\{(X_{fG})^N\}_{N \vartriangleleft_o G}$, it is not possible, 
in general, to 
assemble together the maps $\lambda^N_{N'}$ to 
form, in a non-degenerate way, a diagram 
$\{X^{hN}\}_{N \vartriangleleft_o G}$ of $G$-spectra 
and $G$-equivariant maps. Thus, in general, it is not possible to form, for example, 
$\colim_{N \vartriangleleft_o G} L_k(X^{hN})$ as a discrete $G$-spectrum 
in the desired way. Therefore, to get around this problem, we make the 
following considerations. 
\begin{Def}\label{simplify}
Let $X$ be a discrete $G$-spectrum. If $H$ is a closed subgroup of $G$ 
and $V$ is 
an open normal subgroup of $H$, then we define 
\[X(H, V) = L_k((X_{fH})^V)\] 
and 
\[X(k, H, V) = L_k((X_{f_kH})^V).\] 
\end{Def}
\par
Since $(X_{fH})^V$ has an $H/V$-action and $L_k(-)$ is a functor, 
$X(H,V)$ is an $H/V$-spectrum. Since the finite group $H/V$ is naturally 
discrete, the 
canonical map $H \rightarrow H/V$ makes 
$X(H,V)$ a discrete $H$-spectrum. Thus, 
\[\colim_{V \vartriangleleft_o H} X(H,V) = 
\colim_{V \vartriangleleft_o H}  L_k((X_{fH})^V)\] 
is also a discrete $H$-spectrum. 
The same argument shows that 
\[\colim_{V \vartriangleleft_o H} X(k, H, V) = 
\colim_{V \vartriangleleft_o H}  L_k((X_{f_kH})^V)\] is 
a discrete $H$-spectrum. Since the output of the functor $L_k(-)$ is always a 
fibrant spectrum, $\colim_{V \vartriangleleft_o H} X(H, V)$ and 
$\colim_{V \vartriangleleft_o H} X(k, H, V)$ are fibrant spectra.
\par
By Lemma \ref{open}, the spectrum $X_{f_kH}$ in Definition \ref{simplify} 
is fibrant in $(\Sigma\mathrm{Sp}_V)_k,$ so that $(X_{f_kH})^V$ is already 
$k$-local. However, in Definition \ref{simplify}, 
we are interested instead in the fibrant replacement 
$L_k((X_{f_kH})^V),$ because there is an $H/V$-equivariant 
map \[L_k((f_X^H)^V) \: X(H,V) 
\rightarrow X(k,H,V)\] that will often be used later.  
\par
Let $N$ be an open normal subgroup of $G$. 
The map $\lambda_{N,G} \: X_{fG} \rightarrow X_{fN}$, a weak equivalence 
between fibrant objects in $\Sigma\mathrm{Sp}_N$, induces the weak equivalence 
$(X_{fG})^N \rightarrow (X_{fN})^N = X^{hN}$, and hence, there is a 
weak equivalence 
\[X(G,N) = L_k((X_{fG})^N) \rightarrow L_k(X^{hN}).\] Thus, we can think of 
the discrete $G$-spectrum 
$\colim_{N \vartriangleleft_o G} X(G,N)$ as a replacement for the 
typically undefinable 
object $\colim_{N \vartriangleleft_o G} L_k(X^{hN})$ that was discussed 
just before Definition \ref{simplify}. 
\par
Notice that 
the isomorphism $X \cong \smash{\displaystyle\colim_{V \vartriangleleft_o H}} \, X^V$ yields 
a commutative diagram 
\begin{equation}\zig\label{picture}
\xymatrix{X \ar[r] \ar[d] & 
\smash{\displaystyle{\colim_{V \vartriangleleft_o H}}} \, X(k, H, V) \\
\smash{\displaystyle{\colim_{V \vartriangleleft_o H}}} \, X(H,V) 
\ar@/_1.3pc/[ur]_(.75){\displaystyle{\colim_{V \vartriangleleft_o H} L_k((f_X^H)^V)}} &}
\end{equation} 
in $\Sigma\mathrm{Sp}_H$. The result below considers a case where the 
diagonal map in (\ref{picture}) is a weak equivalence.
\begin{Lem}\label{lemma}
Let $G$ have finite vcd. If $X$ is a discrete $G$-spectrum that 
is $T$-local, then the map 
\[\colim_{N \vartriangleleft_o G} L_k((f_X^G)^N) \: 
\colim_{N \vartriangleleft_o G} \, X(G, N) \overset{\simeq}{\longrightarrow} 
\colim_{N \vartriangleleft_o G} \, X(k, G, N)\] 
is a weak equivalence in $\Sigma\mathrm{Sp}_G$.
\end{Lem} 
\begin{proof}
Let $N$ be an open normal subgroup of $G$. 
By Lemma \ref{fibrant}, $(X_{fG})_{f_kN}$ is fibrant in $\Sigma\mathrm{Sp}_N$, 
giving the commutative diagram 
\[\xymatrix{
X_{fG} \ar[r] \ar[d] & (X_{fG})_{f_kN} \\
(X_{fG})_{fN} \ar[ur] &}\] in $\Sigma\mathrm{Sp}_N$. 
The $N$-fixed points of this diagram give 
the commutative diagram 
\begin{equation}\zig\label{diagramone}
\xymatrix{
(X_{fG})^N \ar[r] \ar[d] & ((X_{fG})_{f_kN})^N \\
((X_{fG})_{fN})^N \ar[ur] &}.
\end{equation}
\par
Since 
$X_{fG} \rightarrow (X_{fG})_{fN}$ is a weak equivalence between fibrant 
objects in $\Sigma\mathrm{Sp}_N$, the vertical map 
$(X_{fG})^N \rightarrow ((X_{fG})_{fN})^N$ in (\ref{diagramone}) 
is a weak equivalence, and hence, a $k$-local 
equivalence. Since $N$ has finite vcd (because $G$ has finite vcd) and 
$X_{fG}$ is $T$-local (because $X$ is $T$-local), 
Theorem \ref{key} implies that the diagonal map in (\ref{diagramone}) 
is a $k$-local equivalence. 
Thus, the 
horizontal map in (\ref{diagramone}) 
is also a $k$-local equivalence. This last $k$-local equivalence 
is the top edge in the commutative diagram 
\begin{equation}\zig\label{diagramtwo}
\xymatrix{
(X_{fG})^N \ar[r] \ar[d]_-{(f_X^G)^N} & ((X_{fG})_{f_kN})^N  \ar[d] \\
(X_{f_kG})^N  \ar[r]  & ((X_{f_kG})_{f_kN})^N.}
\end{equation}
\par
Since $X_{f_kG}$ is fibrant in 
$(\Sigma\mathrm{Sp}_N)_k$, the map 
$X_{f_kG} \rightarrow (X_{f_kG})_{f_kN}$ is a weak equivalence between fibrant 
objects in $(\Sigma\mathrm{Sp}_N)_k$, so that the bottom edge in (\ref{diagramtwo}) 
is a $k$-local equivalence. Also, since $f_X^G$ 
is a 
$k$-local equivalence, the map 
\[ (f_X^G)_{f_kN} \: (X_{fG})_{f_kN} \rightarrow (X_{f_kG})_{f_kN}\]
is a weak equivalence between fibrant objects in $(\Sigma\mathrm{Sp}_N)_k$, 
and hence, the right edge in (\ref{diagramtwo}) is a 
$k$-local equivalence. 
Therefore, the left edge in (\ref{diagramtwo}), $(f_X^G)^N$, 
is a $k$-local equivalence.
\par 
Applying $L_k(-)$ to the $k$-local equivalence $(f_X^G)^N$ 
yields the $k$-local equivalence 
\[
L_k((f_X^G)^N) \: X(G,N) = L_k((X_{fG})^N) \rightarrow L_k((X_{f_kG})^N) = X(k,G,N).
\] 
Since $L_k((f_X^G)^N)$ is a $k$-local equivalence between $k$-local spectra, it is 
a weak equivalence, and since it is a weak equivalence between fibrant 
spectra, we can conclude that the desired map 
$\smash{\displaystyle\colim_{N \vartriangleleft_o G}} 
\, L_k((f_X^G)^N)$ is a weak equivalence of spectra.  
\end{proof}
\par
The following result, which will be used in the proof of Corollary 
\ref{main}, is useful for understanding the relationship between 
Theorem \ref{closed} and its simplification (in certain cases) in Corollary 
\ref{main}.
\begin{Lem}\label{lemmasecond}
Let $G$ be a profinite group and $H$ an open subgroup of $G$. 
If $X$ is a discrete $G$-spectrum, then there is a weak equivalence 
\[\colim_{V \vartriangleleft_o H}  X(k, H, V) \overset{\simeq}{\longrightarrow} 
\colim_{N \vartriangleleft_o G}  X(k, G, N)\] 
in $\Sigma\mathrm{Sp}_H$.
\end{Lem}
\begin{proof} 
Recall that if $K$ is a profinite group and if $U$ is an open 
subgroup of $K$, then $U$ 
contains a subgroup $U_K$ that is open and normal in $K$ (see, 
for example, \cite[Lemma 0.3.2]{Wilson}). This fact implies that if 
$V$ is an open normal subgroup of $H$, then, since $V$ is an open 
subgroup of $G$, $V$ contains a subgroup $V_G$ that is open and normal 
in $G$. Thus, $\{ N \, | \, N \vartriangleleft_o G, \, N < H \, \}$ is a cofinal subcollection 
of $\{ V \, | \, V \vartriangleleft_o H \, \},$ giving the isomorphism 
\[c_1 \: \colim_{V \vartriangleleft_o H} X(k,H,V) = 
\colim_{V \vartriangleleft_o H} L_k((X_{f_kH})^V) \, 
\smash{\overset{\cong}{\longrightarrow}} 
\colim_{N \vartriangleleft_o G, \, N < \, H} L_k((X_{f_kH})^N).\] 
\par
By Lemma \ref{open}, $X_{f_kG}$ is fibrant in $(\Sigma\mathrm{Sp}_H)_k$, 
so that there is a weak equivalence 
\[\lambda_{k, G, H} \: X_{f_kH} \rightarrow X_{f_kG}\] 
in $(\Sigma\mathrm{Sp}_H)_k$. If $N$ 
is an open normal subgroup of $G$ that is contained in $H$, then 
$\lambda_{k, G, H}$ is a weak equivalence between 
fibrant objects in $(\Sigma\mathrm{Sp}_N)_k$, so that 
the map \[(\lambda_{k, G, H})^N \: (X_{f_kH})^N \rightarrow (X_{f_kG})^N\] 
is a weak equivalence in 
$(\Sigma\mathrm{Sp})_k$. Hence, 
the map 
\[L_k((\lambda_{k, G, H})^N) \: 
X(k, H, N) = L_k((X_{f_kH})^N) \rightarrow L_k((X_{f_kG})^N) = X(k, G, N)\] 
is a weak equivalence 
in $(\Sigma\mathrm{Sp})_k$ between $k$-localized objects, and thus, it is a 
weak equivalence between fibrant objects in $\Sigma\mathrm{Sp}$. Therefore, 
the filtered colimit 
\[\colim_{N \vartriangleleft_o G, \, N < \, H} L_k((\lambda_{k, G, H})^N) \: 
\colim_{N \vartriangleleft_o G, \, N < \, H} X(k,H,N) \overset{\simeq}{\longrightarrow} 
\colim_{N \vartriangleleft_o G, \, N < \, H} X(k,G,N)\] is also 
a weak equivalence of 
spectra.
\par
Now let $N$ be any open normal subgroup of $G$: $N \cap H$ is an 
open subgroup of $G$, and hence, $N \cap H$ contains 
a subgroup $(N \cap H)_G$ that is open and normal in $G$. Thus, the collection 
$\{ N \, | \, N \vartriangleleft_o G, \, N < H \, \}$ is a cofinal subcollection of 
$\{ N \, | \, N \vartriangleleft_o G \, \}$. This implies that there is an isomorphism 
\[c_2 \: \colim_{N \vartriangleleft_o G, \, N < \, H} X(k,G,N) \, 
\smash{\overset{\cong}{\longrightarrow}} \colim_{N \vartriangleleft_o G } X(k,G,N),\] 
so that the composition 
\[c_2 \circ 
\bigl(\colim_{N \vartriangleleft_o G, \, N < \, H} L_k((\lambda_{k, G, H})^N)\bigr)
\circ c_1 \: \colim_{V \vartriangleleft_o H} X(k,H,V) \overset{\simeq}{\longrightarrow} 
\colim_{N \vartriangleleft_o G} X(k,G,N)\] is a weak equivalence. 
\end{proof}
\section{Obtaining $k$-local spectra from discrete $G$-spectra}\label{sectionfour}
\par
As in Section \ref{useful}, we continue to let $H$ be a closed 
subgroup of a profinite group $G$, and, as usual, we let 
$X$ be a discrete $G$-spectrum. 
Then the following result says that under certain conditions, 
the two-step process of taking the $H$-homotopy fixed points of $X$ and 
then $k$-localizing can be realized 
somewhat more directly by 
just taking the $H$-homotopy fixed points of 
the discrete $H$-spectrum 
$\smash{\displaystyle{\colim_{V \vartriangleleft_o H}}} \, X(H,V)$, so 
that no subsequent $k$-localization is needed. 
\begin{Thm}\label{closed}
Let $G$ have finite vcd and suppose that 
$X$ is a discrete $G$-spectrum that is $T$-local. If $H$ is a closed 
subgroup of $G$, then there is a zigzag of weak equivalences
\begin{equation}\zig\label{closedequation}
L_k(X^{hH}) \overset{\simeq}{\longrightarrow}
(\colim_{V \vartriangleleft_o H} X(k,H,V))^{hH}
\overset{\simeq}{\longleftarrow} 
(\colim_{V \vartriangleleft_o H} X(H,V))^{hH}.
\end{equation}
\end{Thm}
\begin{proof}
Since $G$ has finite vcd, $H$ does too, and hence, Theorem 
\ref{key} implies that the map 
\[f_1 \: L_k(X^{hH}) \overset{\simeq}{\longrightarrow} X^{h_kH}\] is a 
weak equivalence. Also, since $X_{f_kH}$ is fibrant 
in $\Sigma\mathrm{Sp}_H$ (by Lemma \ref{fibrant}), the map 
$X_{f_kH} \rightarrow (X_{f_kH})_{fH}$ is a 
weak equivalence between fibrant objects in $\Sigma\mathrm{Sp}_H$, 
so that the induced map
\[f_2 \: X^{h_kH} = (X_{f_kH})^H 
\overset{\simeq}{\longrightarrow} ((X_{f_kH})_{fH})^H 
= (X_{f_kH})^{hH}\] is a weak equivalence in $\Sigma\mathrm{Sp}$. 
Therefore, composing the weak equivalences $f_1, f_2$ 
gives the weak equivalence 
\[\widehat{f} = f_2 \circ f_1 
\: L_k(X^{hH}) \overset{\simeq}{\longrightarrow} (X_{f_kH})^{hH}.\]
\par
Let $V$ be an open normal subgroup of $H$ and consider the trivial 
cofibration 
\[g_H^V \: (X_{f_kH})^V \rightarrow L_k((X_{f_kH})^V) = X(k,H,V)\] 
in $(\Sigma\mathrm{Sp})_k$ 
given by the fibrant replacement $L_k(-)$. Notice that 
the map $g_H^V$ is 
also $H/V$-equivariant. By Lemma \ref{open}, 
$X_{f_kH}$ is fibrant in $(\Sigma\mathrm{Sp}_V)_k,$ so that 
$(X_{f_kH})^V$ is $k$-local, and hence, $g_H^V$ is a $k$-local 
equivalence between $k$-local spectra, so that $g_H^V$ is a 
weak equivalence of spectra. Since $X_{f_kH}$ 
is fibrant in $(\Sigma\mathrm{Sp}_V)_k,$ it is also fibrant in 
$\Sigma\mathrm{Sp}_V$, so that $(X_{f_kH})^V$ is a fibrant spectrum. This 
implies that $g_H^V$ is a weak equivalence between fibrant spectra. Therefore, 
there is a weak equivalence 
\[g \: X_{f_kH} \overset{\cong}{\longrightarrow} 
\colim_{V \vartriangleleft_o H} (X_{f_kH})^V 
\overset{\simeq}{\longrightarrow} \colim_{V \vartriangleleft_o H} X(k,H,V)\] 
in $\Sigma\mathrm{Sp}_H$. 
\par
The composition of weak equivalences
\[(g)^{hH} \circ \widehat{f} \: L_k(X^{hH}) \overset{\simeq}{\longrightarrow} 
(\colim_{V \vartriangleleft_o H} X(k,H,V))^{hH}\] gives the first weak equivalence in 
(\ref{closedequation}). Since $H$ has finite vcd, the second weak equivalence in 
(\ref{closedequation}) 
is an immediate consequence of Lemma \ref{lemma}.
\end{proof}
\par
Notice that in Theorem \ref{closed}, the discrete $H$-spectrum 
$\colim_{V \vartriangleleft_o H} X(H,V)$, whose $H$-homotopy fixed 
points give $L_k(X^{hH})$, depends on $H$ for its construction. However, 
the result below shows that if $H$ is open in $G$, then 
$\colim_{V \vartriangleleft_o H} X(H,V)$ can be replaced with the 
discrete $G$-spectrum 
$\colim_{N \vartriangleleft_o G} X(G,N)$, which is independent of $H$.
\begin{Cor}\label{main}
Let $G$ have finite vcd and suppose that $X$ is a discrete 
$G$-spectrum that is $T$-local. If 
$H$ is an open subgroup of $G$, then there is a zigzag of 
weak equivalences
\begin{equation}\zig\label{mainequation}
L_k(X^{hH}) \overset{\simeq}{\longrightarrow}
(\colim_{N \vartriangleleft_o G}  X(k, G, N))^{hH}
\overset{\simeq}{\longleftarrow} 
(\colim_{N \vartriangleleft_o G} X(G,N))^{hH}.
\end{equation}
\end{Cor}
\begin{proof}
By Lemma \ref{lemmasecond} and Theorem \ref{closed}, there are weak 
equivalences 
\[L_k(X^{hH}) \overset{\simeq}{\longrightarrow} 
(\colim_{V \vartriangleleft_o H}  X(k, H, V))^{hH}
\overset{\simeq}{\longrightarrow} 
(\colim_{N \vartriangleleft_o G}  X(k, G, N))^{hH},\] giving the first weak 
equivalence in (\ref{mainequation}). The second weak equivalence in 
(\ref{mainequation}) is an immediate consequence of Lemma \ref{lemma}.
\end{proof} 
\section{A few applications to profinite Galois extensions}
\par
As in the Introduction, 
let $E = \colim_\alpha (E_\alpha)_{fGA}$ be a consistent 
$k$-local profinite $G$-Galois 
extension of $A$ of finite vcd.
As noted in \cite[Proposition 6.2.3]{joint}, 
$E$ is a discrete $G$-spectrum, and, by 
\cite[Remark 6.2.2]{joint}, $E$ is not necessarily $k$-local, but it is 
$T$-local. 
\par
In the theory of $k$-local profinite Galois extensions, the $k$-localization 
plays a very important role. In particular, most results about these extensions 
are only known to be true when $L_k(-)$ is applied to the objects involved. This 
is evident from the near ubiquity of ``$(-)_k$" in 
\cite[Sections 6.2, 6.3, 7]{joint}. However, the 
next result, which shows that $E^{hU}$ is $k$-local for every open subgroup $U$ 
of $G$, gives an interesting example of when the $k$-localization is not 
necessary.
\begin{Thm}\label{galois}
Let $E = \colim_\alpha (E_{\alpha})_{fGA}$ 
be a consistent $k$-local profinite $G$-Galois extension of 
$A$ of finite vcd. If $U$ is an open subgroup of $G$, then  
\[L_k(E^{hU}) \simeq E^{hU}.\] 
\end{Thm}
\begin{proof}
By Corollary \ref{main}, we have
\[L_k(E^{hU}) \simeq 
\bigl(\colim_{N \vartriangleleft_o G} L_k((E_{fG})^N)\bigr)^{hU} 
\cong 
\bigl(\colim_\alpha L_k((E_{fG})^{U_\alpha})\bigr)^{hU},\] 
where the isomorphism is due to the 
fact that $\{U_\alpha\}_\alpha$ is a cofinal collection of 
open normal subgroups of $G$.
\par
For each $\alpha$, the canonical projection $G \rightarrow G/U_\alpha$ makes 
the commutative $G/U_\alpha$-$A$-algebra $E_\alpha$ a discrete commutative 
$G$-$A$-algebra, with $U_\alpha$ acting 
trivially. Thus, by functoriality, $U_\alpha$ also acts trivially on 
$(E_\alpha)_{fGA}$, so that $(E_\alpha)_{fGA}$ is also 
a $G/U_\alpha$-spectrum. 
Since the canonical map $(E_\alpha)_{fGA} \rightarrow E$ is 
$G$-equivariant, the trivial $U_\alpha$-action on the source 
induces a $G/U_\alpha$-equivariant map 
$(E_\alpha)_{fGA} \rightarrow E^{U_\alpha},$ 
which yields the $G/U_\alpha$-equivariant map
\[\widehat{\varepsilon}_{\scriptscriptstyle{\alpha}} \: 
(E_\alpha)_{fGA} \rightarrow E^{U_\alpha} \rightarrow (E_{fG})^{U_\alpha} 
\rightarrow L_k((E_{fG})^{U_\alpha}).\] 
Now \cite[Lemma 6.3.6]{joint} shows that there are equivalences
\[\xymatrix{E_\alpha \ar[r]^-{\varepsilon_{\scriptscriptstyle{\alpha}}}_-\simeq 
& L_k((E_{fGA})^{U_\alpha}) \simeq L_k(E^{hU_\alpha}),}\] 
and it is easy to see that the map 
$\widehat{\varepsilon}_{\scriptscriptstyle{\alpha}}$ is a 
simple variant of the weak equivalence 
$\varepsilon_{\scriptscriptstyle{\alpha}}$ 
(see the discussion just before \cite[Lemma 6.3.6]{joint}). Therefore, 
the weak equivalence ${(E_{fG})^{U_\alpha} 
\overset{\simeq}{\longrightarrow} E^{hU_\alpha}}$ 
implies that the map $\widehat{\varepsilon}_{\scriptscriptstyle{\alpha}}$ is a 
weak equivalence. 
\par
Since $(E_\alpha)_{fGA}$ is a 
positive fibrant discrete $G$-spectrum (by \cite[Theorem 5.2.3]{joint}), it is a 
positive fibrant spectrum, by \cite[Corollary 5.3.3]{joint}. Also, because 
$L_k((E_{fG})^{U_\alpha})$ is fibrant in $\Sigma\mathrm{Sp}$, it is a positive 
fibrant spectrum. Thus, $\widehat{\varepsilon}_{\scriptscriptstyle{\alpha}}$ 
is a weak equivalence in $\Sigma\mathrm{Sp}$ 
between positive fibrant spectra, and hence,
\[\colim_\alpha \ \negthinspace \widehat{\varepsilon}_{\scriptscriptstyle{\alpha}} \: 
E = \colim_\alpha (E_\alpha)_{fGA} 
\overset{\simeq}{\longrightarrow} 
\colim_\alpha L_k((E_{fG})^{U_\alpha})\] is a weak equivalence 
in $\Sigma\mathrm{Sp}_G$, by \cite[Lemma 5.3.1]{joint}. Then we have 
\[L_k(E^{hU}) \simeq 
\bigl(\colim_\alpha L_k((E_{fG})^{U_\alpha})\bigr)^{hU} 
\overset{\simeq}{\longleftarrow} E^{hU},\] giving the desired result. 
\end{proof}
\par
By \cite[Corollary 6.3.2, $\negthinspace$ Lemmas 5.2.5, 5.2.6]{joint}, 
there is a zigzag of 
weak equivalences 
\[A \overset{\simeq}{\longrightarrow} 
L_k(E^{h_\mathrm{Alg}G}) \overset{\simeq}{\longleftarrow} 
L_k(E^{h^+G}) \overset{\simeq}{\longrightarrow} 
L_k(E^{hG}),\] where $E^{h_\mathrm{Alg}G}$ is 
the ``homotopy fixed point commutative $A$-algebra" and 
$E^{h^+G}$ is the homotopy fixed point spectrum of $E$ that is formed with 
respect to the positive stable model structure on $\Sigma\mathrm{Sp}_G$ 
(see \cite[Section 5.2]{joint} for more detail). 
Thus, we obtain the equivalence
\begin{equation}\zig\label{consistent}
A \simeq L_k(E^{hG}).
\end{equation} 
However, the next result, which follows immediately from 
Theorem \ref{galois} (by setting $U=G$) and (\ref{consistent}), 
shows that, in fact, no $k$-localization is 
necessary to obtain $A$ from the homotopy fixed points of $E$.
\begin{Cor}\label{A}
If $E$ is a consistent $k$-local profinite $G$-Galois extension of 
$A$ of finite vcd, then  
\[A \simeq E^{hG}.
\]
\end{Cor}
\par
Now let $H$ be a closed subgroup of $G$. 
By Theorem \ref{galois}, $L_k(E^{hH}) \simeq E^{hH}$ whenever 
$H$ is open in $G$. However, this equivalence need not hold for 
all $H$ closed in $G$; this is seen, 
for example, when $H = \{e\}$, since $E^{h\{e\}} \simeq E$ is 
not necessarily $k$-local. However, it follows immediately from Theorem 
\ref{closed} that for every $H$ closed in $G$, a slightly coarser version of 
Theorem \ref{galois} is still valid: by Theorem \ref{closed}, 
\begin{equation}\zig\label{coarse} 
L_k(E^{hH}) \simeq \bigl(\colim_{V \vartriangleleft_o H} L_k((E_{fH})^V)\bigr)^{hH},
\end{equation} 
so that $L_k(E^{hH})$ can still always be obtained by taking the 
$H$-homotopy fixed points of a discrete $H$-spectrum, with no subsequent 
$k$-localization needed.
\par
Notice that in the discrete $H$-spectrum 
$\colim_{V \vartriangleleft_o H} L_k((E_{fH})^V)$ of (\ref{coarse}), whose 
$H$-homotopy fixed points yield $L_k(E^{hH})$, both the indexing set of the 
colimit and $E_{fH}$ depend on $H$. Now we show that when 
the Galois extension $E$ is profaithful (defined below), the second of these 
dependencies can be eliminated.
\begin{Def}[{\cite[Definition 6.2.1]{joint}}]\label{profaithful}
Let $E = \colim_\alpha (E_\alpha)_{fGA}$ 
be a $k$-local profinite $G$-Galois extension of $A$. If the $k$-local 
$G/U_\alpha$-Galois extension $E_\alpha$ is 
$k$-locally 
faithful (that is, following \cite[Definition 4.3.1]{rognes}, 
if $M$ is an $A$-module and $L_k(M \wedge_A E_\alpha) \simeq \ast,$ 
then $L_k(M) \simeq \ast$) for each $\alpha$, then 
$E$ is a {\em profaithful $k$-local} extension.
\end{Def}
\begin{Thm}\label{lasttheorem}
Let $E$ be a consistent profaithful $k$-local profinite $G$-Galois extension of 
$A$ of finite vcd. If $H$ is a closed subgroup of $G$, then 
\begin{equation}\zig\label{almostdone}
L_k(E^{hH}) \simeq 
\bigl(\colim_{V \vartriangleleft_o H} L_k((E_{fG})^V)\bigr)^{hH}.
\end{equation}
\end{Thm}
\begin{Rk}
Let $E$ and $H$ 
be as in Theorem \ref{lasttheorem}. By (\ref{coarse}), $L_k(E^{hH})$ 
can be obtained by taking the homotopy fixed points of the discrete $H$-spectrum 
\[L = \colim_{N \vartriangleleft_o G} L_k((E_{fH})^{H \cap N}) \cong 
\colim_{V \vartriangleleft_o H} L_k((E_{fH})^V),\] where the isomorphism 
uses that $\{H \cap N \, | \, N \vartriangleleft_o G \, \}$ is a cofinal subcollection of 
$\{V \, | \, V \vartriangleleft_o H \, \}.$ The proof of \cite[Theorem 7.2.1, (1)]{joint} shows 
that $L$ is a consistent profaithful $k$-local profinite 
$H$-Galois extension of $L_k(E^{hH})$ of finite vcd. 
Also, the proof below of Theorem \ref{lasttheorem} 
shows that there is a weak equivalence 
\begin{equation}\zig\label{secondmap}
\colim_{V \vartriangleleft_o H} L_k((E_{fG})^V) \overset{\simeq}{\longrightarrow} L
\end{equation} 
in $\Sigma\mathrm{Sp}_H$, where the source of this weak equivalence is the 
discrete $H$-spectrum that appears on the right-hand side in (\ref{almostdone}). 
The proof of \cite[Theorem 7.2.1, (1)]{joint} shows 
that there is at least a zigzag of 
weak equivalences between the source and target of the map in 
(\ref{secondmap}), so the 
main value of the proof of Theorem \ref{lasttheorem} below is, 
in addition to pointing out the relevant details from \cite{joint}, to show explicitly that 
there is a single weak equivalence. 
\end{Rk}
\begin{proof}[{Proof of Theorem \ref{lasttheorem}}]
Let $V$ be an open normal subgroup of $H$. The map 
\[\xymatrix{{\smash{\displaystyle{\colim_{V < \, U <_o G}}}} (E_{fG})^{U} 
\cong (E_{fG})^V \ar[r]^-{\psi(V)} & ((E_{fG})_{fH})^V}\] induces the map 
\[\widehat{\psi}_V \: 
L_k\bigl(\colim_{V < \, U <_o G} (E_{fG})^{U}\bigr) \rightarrow 
L_k\bigl(((E_{fG})_{fH})^V\bigr).\] By \cite[Theorem 7.1.1]{joint}, 
$\widehat{\psi}_V$ is a weak equivalence, and hence, 
the $H/V$-equivariant map 
\[L_k(\psi(V)) \: L_k((E_{fG})^V) \overset{\simeq}{\longrightarrow} 
L_k\bigl(((E_{fG})_{fH})^V\bigr)\] is a 
weak equivalence.
\par
Since the composition $E \rightarrow E_{fG} \rightarrow 
(E_{fG})_{fH}$ is a trivial cofibration in $\Sigma\mathrm{Sp}_H$, 
there is a weak equivalence 
$(E_{fG})_{fH} \overset{p}{\rightarrow} E_{fH}$ in $\Sigma\mathrm{Sp}_H$. 
Hence, $p$ is a weak equivalence between fibrant objects, 
in $\Sigma\mathrm{Sp}_V$, so that 
the $H/V$-equivariant map \[p^V \: ((E_{fG})_{fH})^V \overset{\simeq}{\longrightarrow} 
(E_{fH})^V\] is a 
weak equivalence. Therefore, the composition 
\[L_k(p^V) \circ L_k(\psi(V)) \: 
L_k((E_{fG})^V) \overset{\simeq}{\longrightarrow} 
L_k\bigl(((E_{fG})_{fH})^V\bigr) \overset{\simeq}{\longrightarrow} 
L_k((E_{fH})^V),\] which 
is $H/V$-equivariant, is a weak equivalence, and since its source and 
target are fibrant spectra, there is a weak equivalence 
\[\colim_{V \vartriangleleft_o H} L_k((E_{fG})^V) \overset{\simeq}{\longrightarrow} 
\colim_{V \vartriangleleft_o H} L_k((E_{fH})^V)\] in $\Sigma\mathrm{Sp}_H$, giving 
\[\bigl(\colim_{V \vartriangleleft_o H} L_k((E_{fG})^V)\bigr)^{hH} 
\overset{\simeq}{\longrightarrow} 
\bigl(\colim_{V \vartriangleleft_o H} L_k((E_{fH})^V)\bigr)^{hH} \simeq L_k(E^{hH}),\] 
where the last equivalence is from (\ref{coarse}).   
\end{proof}


\begin{thebibliography}{10}

\bibitem{joint}
Mark Behrens and Daniel~G. Davis.
\newblock The homotopy fixed point spectra of profinite {G}alois extensions.
\newblock {\em Trans. Amer. Math. Soc.}, 362(9):4983--5042, 2010.

\bibitem{Bousfieldlocal}
A.~K. Bousfield.
\newblock The localization of spectra with respect to homology.
\newblock {\em Topology}, 18(4):257--281, 1979.

\bibitem{derivedcompletion}
Gunnar Carlsson.
\newblock Derived completions in stable homotopy theory.
\newblock {\em J. Pure Appl. Algebra}, 212(3):550--577, 2008.

\bibitem{cts}
Daniel~G. Davis.
\newblock Homotopy fixed points for {$L\sb {K(n)}(E\sb n\wedge X)$} using the
  continuous action.
\newblock {\em J. Pure Appl. Algebra}, 206(3):322--354, 2006.

\bibitem{thesis}
Daniel~G. Davis.
\newblock The {L}ubin-{T}ate spectrum and its homotopy fixed point spectra.
\newblock Thesis, 106 pp., Northwestern University, May 9, 2003, available at
  www.ucs.louisiana.edu/{$\sim$}dxd0799.

\bibitem{DH}
Ethan~S. Devinatz and Michael~J. Hopkins.
\newblock Homotopy fixed point spectra for closed subgroups of the {M}orava
  stabilizer groups.
\newblock {\em Topology}, 43(1):1--47, 2004.

\bibitem{hGal}
Paul~G. Goerss.
\newblock Homotopy fixed points for {G}alois groups.
\newblock In {\em The \v Cech Centennial (Boston, MA, 1993)}, pages 187--224.
  Amer. Math. Soc., Providence, RI, 1995.

\bibitem{eo2homotopy}
Mike Hopkins and Mark Mahowald.
\newblock From elliptic curves to homotopy theory, 30 pp., July 20, 1998,
  available online at the Hopf Topology Archive.

\bibitem{HoveyCech}
Mark Hovey.
\newblock Bousfield localization functors and {H}opkins' chromatic splitting
  conjecture.
\newblock In {\em The \v Cech Centennial (Boston, MA, 1993)}, pages 225--250.
  Amer. Math. Soc., Providence, RI, 1995.

\bibitem{HSS}
Mark Hovey, Brooke Shipley, and Jeff Smith.
\newblock Symmetric spectra.
\newblock {\em J. Amer. Math. Soc.}, 13(1):149--208, 2000.

\bibitem{rognes}
John Rognes.
\newblock Galois extensions of structured ring spectra.
\newblock In {\em Galois extensions of structured ring spectra/{S}tably
  dualizable groups, {$\mathrm{in}$} Mem. Amer. Math. Soc., {$\mathrm{vol. \
  192 \ (898)}$}}, pages 1--97, 2008.

\bibitem{Wilson}
John~S. Wilson.
\newblock {\em Profinite groups}.
\newblock The Clarendon Press, Oxford University Press, New York, 1998.

\end{thebibliography}

\end{document}